\documentclass[12pt,reqno]{amsart}
\usepackage{diagrams}
\usepackage{amssymb}
\usepackage{graphicx}
\usepackage{amsmath}
\usepackage{a4wide}
\usepackage{version}
\usepackage{mathrsfs}
\usepackage{tikz}
\usepackage{mathtools}
\usetikzlibrary{arrows,decorations.markings, matrix}
\usepackage{tikz-cd}
\usepackage{hyperref}
\usepackage{float}
\usepackage[all]{xy}

\numberwithin{equation}{section}

\newtheorem{thm}{Theorem}[subsection]
\newtheorem{lem}[thm]{Lemma}
\newtheorem{prop}[thm]{Proposition}
\newtheorem{cor}[thm]{Corollary}

\theoremstyle{definition}

\newtheorem{rmk}[thm]{Remark}

\renewcommand{\k}{\mathbf{k}}

\newcommand{\id}{\operatorname{id}}

\newcommand{\ov}{\overline}
\newcommand{\we}{\wedge}

\newcommand{\rk}{\operatorname{rk}}

\newcommand{\XX}{{\mathcal X}}

\newcommand{\VV}{{\mathcal V}}
\renewcommand{\SS}{{\mathcal S}}

\newcommand{\Ext}{\operatorname{Ext}}
\newcommand{\End}{\operatorname{End}}

\renewcommand{\a}{\alpha}
\renewcommand{\b}{\beta}

\newcommand{\la}{\lambda}

\renewcommand{\P}{{\mathbb P}}

\newcommand{\wt}{\widetilde}
\newcommand{\ot}{\otimes}

\newcommand{\sub}{\subset}

\newcommand{\vol}{\operatorname{vol}}

\newcommand{\GL}{\operatorname{GL}}

\newcommand{\Bun}{\operatorname{Bun}}
\renewcommand{\k}{\mathbf{k}}

\renewcommand{\mod}{\operatorname{mod}}

\newcommand{\OO}{{\mathcal O}}

\newcommand{\Si}{\Sigma}

\newcommand{\II}{{\mathcal I}}

\newcommand{\hra}{\hookrightarrow}
\newcommand{\lan}{\langle}
\newcommand{\ran}{\rangle}

\renewcommand{\P}{{\mathbb P}}

\newcommand{\si}{\sigma}

\newcommand{\Pic}{\operatorname{Pic}}

\newcommand{\de}{\delta}

\renewcommand{\ker}{\operatorname{ker}}
\newcommand{\im}{\operatorname{im}}

\newcommand{\tr}{\operatorname{tr}}

\newcommand{\Sec}{\operatorname{Sec}}

\newcommand{\fg}{{\mathfrak g}}
\newcommand{\fm}{{\mathfrak m}}

\title[]{Compatible Feigin-Odesskii Poisson brackets}
\author{Nikita Markarian}
\author{Alexander Polishchuk}
\thanks{A.P. is supported in part by the NSF grant DMS-2001224, 
and within the framework of the HSE University Basic Research Program and by the Russian Academic Excellence Project `5-100'.}

\begin{document}

\begin{abstract}
We prove that several Feigin-Odesskii Poisson brackets associated with normal elliptic curves in $\P^n$ are compatible if and only if 
they are contained in a scroll or in a Veronese surface in $\P^5$ (with an exception of one case when $n=3$). In the case $n=3$ we determine
the quartic corresponding to the Schouten bracket of two (non-compatible) Poisson brackets associated with normal elliptic curves $E_1$ and $E_2$.
\end{abstract}

\maketitle

\section{Introduction}

We work over an algebraically closed field $\k$ of characteristic $0$. 

This work is a contribution to the study of a remarkable class of Poisson brackets on projective
spaces associated with elliptic curves, that were introduced by Feigin and Odesskii (see \cite{FO95}).
Namely, for every normal elliptic curve $E\sub \P^n$, there is an associated Poisson bracket $\Pi_E$
defined up to rescaling (we refer to these as FO brackets; sometimes they are denoted as $q_{n+1,1}(E)$). 
Setting $L=\OO(1)|_E$ one can identify $\P^n$
with the projective space of extensions of $L$ by $\OO$. Then the bracket $\Pi_E$ can be defined and studied in
terms of geometry of such extensions on $E$.

Recall that two Poisson structures $\Pi_1,\Pi_2$ on the same space $X$ are called compatible if every linear combination $\la_1\Pi_1+\la_2\Pi_2$ is still a Poisson
structure. This is equivalent to the identity $[\Pi_1,\Pi_2]=0$, where we use the Schouten bracket of bivectors. More generally, one can consider larger linear subspaces of Poisson bivectors.
In the work \cite{OW} Odesskii and Wolf gave a construction of $9$-dimensional subspaces of Poisson structures on $\P^n$ whose general member is a bracket $\Pi_E$ for some $E\sub \P^n$. 
In \cite{HP-bih} this construction was interpreted geometrically in terms of families of anticanonical divisors on Hirzebruch surfaces, and extended to give new examples of compatible
Poisson structures. In the present work we show that these constructions are the only way to produce compatible FO brackets of type $q_{n+1,1}(E)$, with one exception occurring for $n=3$.

First, let us consider the case of FO brackets associated with normal elliptic curves in $\P^3$.

\medskip

\noindent
{\bf Theorem A}.
{\it For a collection of normal elliptic curves $(E_i)_{i\in I}$ in $\P^3=\P V$, the brackets $(\Pi_{E_i})$ are compatible if and only if either all $E_i$ lie on one quadric surface (possibly singular),
or there exists a $3$-dimensional subspace $W\sub S^2V^*$ such that every $E_i$ is the intersection of two quadrics in $W$.
}

\medskip

Theorem A has a natural generalization to FO brackets in higher-dimensional projective spaces.
Let us fix $n\ge 4$. With every normal elliptic curve $E$ in $\P^n$ we can associate a $1$-dimensional family $\SS_E$ of rational normal scrolls $S(r,r)\sub \P^n$ if $n=2r+1$,
with several exceptional members of type $S(r-1,r+1)$ (resp., $S(r-1,r)$ if $n=2r$), which are parametrized by points of $\Pic_2(E)$, see Sec.\ \ref{scrolls-sec} for details.
In the case of $n=5$, each normal elliptic curve $E$ in $\P^5$ is contained in four Veronese surfaces, corresponding to choices of a square root of the line bundle $\OO(1)|_E$ of degree $6$.

\medskip

\noindent
{\bf Theorem B}.
{\it For a collection of normal elliptic curves $(E_i)_{i\in I}$ in 
$\P^n$, where $n\ge 4$, the brackets $(\Pi_{E_i})$
are compatible if and only if 
\begin{itemize}
\item either the corresponding families $(\SS_{E_i})$ have an element in common,
\item or $n=5$ and all $E_i$ are contained in a Veronese surface $\P^2\sub \P^5$.
\end{itemize}
}

\medskip

The family of elliptic curves contained in a scroll is precisely the family of anticanonical divisors producing compatible Poisson brackets of Odesskii-Wolf 
(see \cite[Sec.\ 5.3]{HP-bih}). 
The fact that elliptic curves contained in a Veronese surface in $\P^5$ give rise to compatible Poisson brackets was observed in \cite[Ex.\ 4.6]{HP-bih}.

Let us call an {\it FO-subspace} any linear subspace of compatible Poisson brackets on $\P^n$, whose generic point is of the form $\Pi_E$ for some normal elliptic curve $E$.
Odesskii-Wolf conjectured that their examples of $9$ compatible Poisson brackets on $\P^n$ give maximal linear subspaces of Poisson brackets.
We prove that this is true if $n$ is even, while for odd $n$ we prove maximality of these subspaces among FO-subspaces.

\medskip

\noindent
{\bf Corollary C}. 
{\it (a) Given a rational normal scroll $S=S(s,r)$ in $\P^n$, with $0<s\le r\le s+2$, consider the $9$-dimensional FO-subspace $V_S$ of Poisson brackets on $\P^n$ coming from anticanonical 
divisors on $S$. Then $V_S$ is a maximal subspace of Poisson brackets provided $n$ is even (resp., a maximal FO-subspace if $n$ is odd). 

\noindent
(b) The maximal dimension of a linear subspace of Poisson brackets on $\P^n$, containing some FO bracket $\Pi_E$, is $9$ provided $n$ is even and $n\ge 4$.
For odd $n$, the maximal dimension of an FO-subspace of Poisson brackets on $\P^n$ 
is $9$, for $n\ge 3$, $n\neq 5$, and is $10$ for $n=5$.}
\medskip

The construction of compatible Poisson brackets in \cite{HP-bih} also produces some families of compatible FO brackets of type $q_{n,k}(E)$ (these are associated with
higher rank stable bundles on elliptic curves). It would be interesting to find criteria for compatibility of brackets of type $q_{n,k}(E)$ similar to our Theorems A and B.

Let us return to FO brackets on $\P^3$.
In the case when $\Pi_{E_1}$ and $\Pi_{E_2}$ are not compatible, the Schouten bracket $[\Pi_{E_1}, \Pi_{E_2}]$
is a section of ${\bigwedge}^3T_{\P^3}\simeq \OO_{\P^3}(4)$, so it gives a quartic hypersurface in $\P^3$. We will compute this hypersurface in terms of defining pairs of quadrics for $E_1$ and $E_2$.
Let $\P^3=\P V$, where $V$ is a $4$-dimensional vector space. Consider a natural linear map
$$\Phi:{\bigwedge}^4(S^2V^*)\to \det(V^*)\ot S^4(V^*)$$
that associates with $Q_1\we Q_2\we Q_3\we Q_4$ the polynomial map of degree $4$ with values in $\det(V^*)$, 
$$v\mapsto (\partial_v Q_1)\we(\partial_v Q_2)\we(\partial_v Q_3)\we (\partial_v Q_4).$$
On the other hand, the Poisson bivector associated with $Q_1=Q_2=0$ canonically is an element of $\det(V^*)\ot H^0(\P V,{\bigwedge}^2 T_{\P V})$,
hence the Schouten bracket of two such bivectors is a section of 
$$(\det V^*)^{\ot 2}\ot {\bigwedge}^3T_{\P^3}\simeq (\det V^*)^{\ot 2}\ot\det(V)\ot \OO_{\P V}(4)\simeq \det(V^*)\ot \OO_{\P V}(4).$$

\medskip

\noindent
{\bf Theorem D}. 
{\it One has an equality in $\det(V^*)\ot S^4(V^*)$,
$$[\Pi_{Q_1=Q_2=0},\Pi_{Q_3=Q_4=0}]=4\cdot \Phi(Q_1\we Q_2\we Q_3\we Q_4).$$ 
}

\medskip

The quartic surface $X_{E_1,E_2}\sub \P^3$ given by $[\Pi_{E_1},\Pi_{E_2}]=0$ can be characterized geometrically as follows.
First, with a normal elliptic curve $E\sub \P^3$ and a point $p\in \P^3$, which is not the vertex of one of the four singular quadrics through $E$, we can associate a line $L_E(p)\sub \P^3$ called the {\it polar} line of $p$ with respect to $E$: if $E$ is given by $Q_1=Q_2=0$, then $L_E(p)$ is the intersection of polar planes to $p$ with respect to $Q_1=0$ and $Q_2=0$ (it does not
depend on a choice of $Q_1,Q_2$). Now $X_{E_1,E_2}$ is the closure of the set of points $p$ such that the polars $L_{E_1}(p)$ and $L_{E_2}(p)$ have a nontrivial
intersection (this quartic also contains vertices of singular quadrics through $E_1$ and $E_2$).
Another definition of the polar line to a point $p$ with respect to $E$ (for a generic $p$) is the following. Take two secant lines to $E$, $L_1=\ov{p_1q_1}$ and $L_2=\ov{p_2q_2}$, passing through $p$ (they coincide with two lines through $p$ on the unique quadric $Q$ passing through $E$ and $p$). On each of these two lines $L_i$, choose a point $r_i$, such that
$(p,p_i,q_i,r_i)$ is a harmonic quadruple (have double ratio $-1$). Then the polar line $L_E(p)$ is the line passing through $r_1$ and $r_2$.



\bigskip

\noindent
{\it Acknowledment}. This work was done while both authors were staying at the IHES. We thank this institution for hospitality and excellent working environment.

\section{Some facts about FO brackets}

\subsection{Conormal Lie algebras}

Recall that for any point $x$ of a smooth Poisson variety $(X,\Pi)$, there is a natural Lie algebra structure on the space
$$\fg_{\Pi,x}:=\ker(\Pi_x:T_x^*X\to T_xX)\sub T_x^*X.$$
Namely, if we lift $a,b\in \fg_x\sub \fm_x/\fm_x^2$ to some local functions $\wt{a},\wt{b}\in\fm_x$, then
$\{\wt{a},\wt{b}\}\in \fm_x$ and
$$[a,b]:=\{\wt{a},\wt{b}\}\mod \fm_x^2.$$

Now let $\Pi=\Pi_E$ denote the FO bracket on $\P^n=\P\Ext^1(L,\OO)$, associated with a normal elliptic curve $E\sub \P^n$ (so $L=\OO(1)|_E$). 
It is known (see \cite[Prop.\ 2.3]{HP-bih}) that the rank of $\Pi_x$ is equal to $n+1-\dim\End(V)$, where $V$ is the extension of $L$ by $\OO$ corresponding to $x$.
The next result realizes this equality by an isomorphism of natural Lie algebras.

\begin{prop}\label{conormal-Lie-prop}
For $x\in \Ext^1(L,\OO)\setminus\{0\}$, which is the class of an extension
\begin{equation}\label{Ext-VL-eq}
0\to \OO\rTo{i} V\rTo{p} L\to 0,
\end{equation}
the conormal Lie algebra $\fg_{\Pi,x}$ for the FO-bracket $\Pi=\Pi_E$ is naturally isomorphic to the Lie algebra $\End(V)/\lan\id\ran$. 
\end{prop}

\begin{proof}
Recall (see \cite[Lem.\ 2.1]{HP-bih}) that for $x\in \Ext^1(L,\OO)\setminus\{0\}$ and $s_1,s_2\in\lan x\ran^\perp\sub H^0(L)$, the bivector $\Pi_x$ is given by 
$$s_1\we s_2\mapsto \lan MP(s_1,x,s_2),x\ran,$$
where $MP(s_1,x,s_2)\in H^0(L)/\lan s_1,s_2\ran$ is the Massey product.
Note that we can identify $\lan x\ran^\perp\sub H^0(L)$ with the kernel of the connecting homomorphism $H^0(L)\to H^1(\OO)$ associated with \eqref{Ext-VL-eq},
hence, $\lan x\ran^\perp$ consists of sections liftable to $H^0(V)$. 
By definition, $MP(s_1,x,s_2)$ is the $\lan s_1,s_2\ran$-coset represented by $\wt{s}_1(\wt{s}_2)$, where
$\wt{s}_1:V\to L$ is  a map extending $s_1:\OO\to L$ and $\wt{s}_2\in H^0(V)$ is a lifting of $s_2$. 

Let us consider the map
$$\End(V)/\lan\id\ran \to H^0(L): A\mapsto p\circ A\circ i.$$
We claim that this map is injective. Indeed, if $p\circ A\circ i=0$ for $A\in \End(V)$, then $A\circ i=c\cdot i$ for some constant $c$. Hence, $(A-c\id)\circ i=0$, so $A-c\id$ factors through
a map $L\to V$. But any such map is zero since the extension \eqref{Ext-VL-eq} does not split, so $A\in \lan\id\ran$, which proves our claim.

Now suppose $s_1=p\circ A\circ i$, for some $A\in \End(V)$. Then we can take $\wt{s}_1=p\circ A$. Hence, $\wt{s}_1(\wt{s}_2)=p(A(\wt{s}_2))$ which lies in $\lan x\ran^\perp$
since it is liftable to $H^0(V)$. Thus, we have an inclusion
$$\End(V)/\lan\id\ran\hra \ker(\Pi_x)\sub \lan x\ran^\perp.$$
Since the dimensions are equal, in fact, we get an identification of vector spaces
$$\End(V)/\lan\id\ran\rTo{\sim} \fg_{\Pi,x}=\ker(\Pi_x).$$

It remains to check compatibility of the brackets. Let us fix a pair of endomorphisms $A,A'\in\End(V)$, and let $s=p\circ A\circ i$ and $s'=p\circ A'\circ i$ be the corresponding elements of
$\ker(\Pi_x)$. Let us also fix a vector $v\in \Ext^1(L,\OO)\setminus \lan x\ran$, representing a nontrivial tangent direction to $\lan x\ran$ in the projective space.
Let us consider the corresponding line $x(t)=x+tv$ in $\Ext^1(L,\OO)$, and the induced line $\lan x(t)\ran$ in the projective space.
In order to compute $\lan [s,s'],v\ran$ we have to extend $s$ and $s'$ to sections $s(t)$, $s'(t)$ of the cotangent bundle along the line $\lan x(t)\ran$. Then by definition, we have
$$\lan [s,s'],v\ran=\frac{d}{dt} \lan\Pi, s(t)\we s'(t)\ran |_{t=0}.$$
Furthermore, it is enough to work over dual numbers, where $t^2=0$.
Thus, we can pick any elements $u,u'\in H^0(L)$, such that
$$\lan x,u\ran+\lan v,s\ran=0, \ \lan x,u'\ran+\lan v,s'\ran=0,$$
and set $s(t)=s+tu$, $s'(t)=s'+tu'$. Then 
$$\lan x(t),s(t)\ran=\lan x+tv,s+tu\ran=0, \ \ \lan x(t),s'(t)\ran=0,$$
so $s(t)$ and $s'(t)$ are in sections of the cotangent bundle. Now we have to calculate the derivative at $0$ (i.e., the coefficient of $t$)of
$$\lan\Pi, s(t)\we s'(t)\ran=\lan MP(s+tu,x+tv,s'+tu'),x+tv\ran,$$

To calculate this we use Cech representatives. We cover $E$ with two affine opens $U_1$ and $U_2$, set $U_{12}=U_1\cap U_2$, and realize $x$ and $v$ by
Cech $1$-cocycles $x\in L^{-1}(U_{12})$, $v\in L^{-1}(U_{12})$. The extension $V=V_x$ of $L$ by $\OO$ associated with $x$ is equipped with splitting $\si_a:L\to V$ over each $U_a$, 
$a=1,2$, such that over $U_{12}$ we have 
$$(\si_2-\si_1)(\phi)=i(x\cdot \phi),$$
where $\phi\in L(U_{12})$. 

If we have a section $s\in \lan x\ran^\perp\sub H^0(L)$, then there exist $f_a\in \OO(U_a)$, $a=1,2$, such that
\begin{equation}\label{x-s-fa-relation}
x\cdot s=f_1-f_2.
\end{equation}
Then we can extend $s$ to a section $\a=\a(s,f_\bullet):\OO\to V$ and to a morphism $\b=\b(s,f_\bullet):V\to L$ such that
$p(\a)=s=\b\circ i$, and
$$\a|_{U_a}=i(f_a)+\si_a(s), \ \ \b|_{U_a}(\si_a(\phi))=-f_a\cdot \phi.$$

It is easy to see that an operator $A\in \End(V)$ with $\tr(A)=0$, is determined by the section $s=p\circ A\circ i$, together with
functions $f_a\in \OO(U_a)$ and sections $A_a\in L^{-1}(U_a)$, $a=1,2$, satisfying \eqref{x-s-fa-relation} and
\begin{equation}\label{A2-A1-f-x-eq}
A_2-A_1=(f_1+f_2)\cdot x.
\end{equation}
Namely, the operator $A:V\to V$ is given over $U_i$ by
$$A i(1)=f_a+\si_a(s), \ \ A \si_a(\phi)=i(A_a\cdot\phi)-f_a\cdot \si_a(\phi),$$
for $\phi\in L(U_a)$.

Next, we need to consider the extension $V_{x+tv}$ of $L$ by $\OO$ corresponding to $x+tv$. We denote by $\si_a(t):L\to V_{x+tv}$ the splittings over $U_a$,
$a=1,2$, such that
$$(\si_2(t)-\si_1(t))(\phi)=i((x+tv)\cdot \phi).$$
We start with two traceless operators $A,A'\in \End_0(V)$ and consider the corresponding sections $s,s'\in \lan x\ran^\perp\sub H^0(L)$. Thus, we can assume that we have
the corresponding functions $f_a,f'_a\in \OO(U_a)$ and sections $A_a,A'_a\in L^{-1}(U_a)$ satisfying \eqref{x-s-fa-relation}, \eqref{A2-A1-f-x-eq} and the similar relations
involving $s'$, $f'_a$ and $A'_a$.
We also choose functions $g_a,g'_a\in \OO(U_a)$ such that
\begin{equation}\label{vs-xu-g-relation}
v\cdot s+x\cdot u=g_1-g_2, \ \ v\cdot s'+x\cdot u'=g'_1-g'_2.
\end{equation}
Then we have the following deformed version of \eqref{x-s-fa-relation} over $U_{12}$:
$$(x+tv)\cdot (s+tu)=(f_1+g_1t)-(f_2+g_2t).$$

By definition, the Massey product $MP(s+tu,x+tv,s'+tu')$ is the class of the composition $\b\circ \a'$, where $\a':\OO\to V_{x+tv}$ is the global section projecting to $s'+tu'$,
while $b:V_{x+tv}\to L$ is the morphism extending $s+tu$. As we have seen before, we can define $\a'$ and $\b$ so that over $U_a$ we have
$$a'=i(f'_a+tg'_a)+\si_a(s'+tu'),$$
$$\b(i(f)+\si_a(\phi))=f(s+tu)-(f_a+tg_a)\phi.$$
Hence, 
\begin{align*}
&MP=MP(s+tu,x+tv,s'+tu')=(s+tu)(f'_a+tg'_a)-(s'+tu')(f_a+tg_a)=\\
&sf'_a-s'f_a+t(sg'_a+uf'_a-s'g_a-u'f_a),
\end{align*}
and therefore, considering the coefficient of $t$ in $\lan MP,x+tv\ran$ we get 
$$[s,s']=v\cdot(sf'_a-s'f_a)+x\cdot(sg'_a+uf'_a-s'g_a-u'f_a) \mod \im(\de),$$
where $a=1$ or $2$, and we view the right-hand side as a Cech cohomology class $H^1(\OO)\simeq \k$
(the answers for $a=1$ and $a=2$ differ by a cocycle).
Therefore, we have
$$2[s,s']=v\cdot[s(f'_1+f'_2)-s'(f_1+f_2)]+C \mod \im(\de),$$
where
$$C=x\cdot [s(g'_1+g'_2)+u(f'_1+f'_2)-s'(g_1+g_2)-u'(f_1+f_2)]\mod \im(\de).$$
Note that $u$ and $u'$ are global sections, while $x(f_1+f_2)$ and $x(f'_1+f'_2)$ are coboundaries due to the relation \eqref{A2-A1-f-x-eq}.
Hence, 
$$C=xs(g'_1+g'_2)-xs'(g_1+g_2)=(f_1-f_2)(g'_1+g'_2)-(f'_1-f'_2)(g_1+g_2)\equiv f_1g'_2-f_2g'_1-f'_1g_2+f'_2g_1 \mod \im(\de),$$
where we removed the coboundary terms $f_1g'_1$, etc.

On the other hand, multiplying the first of the relations \eqref{vs-xu-g-relation} with $f'_1+f'_2$, we get
$$f'_2g_1-f'_1g_2\equiv (f'_1+f'_2)(g_1-g_2)=(f'_1+f'_2)vs+(f'_1+f'_2)xu\equiv (f'_1+f'_2)vs \mod \im(\de).$$
Similarly,
$$f_2g'_1-f_1g'_2\equiv (f_1+f_2)vs' \mod \im(\de).$$
Hence, we obtain
$$C\equiv v\cdot[s(f'_1+f'_2)-s'(f_1+f_2)],$$
which gives
$$\lan [s,s'],v\ran=v\cdot[s(f'_1+f'_2)-s'(f_1+f_2)] \mod \im(\de).$$
Finally, using \eqref{x-s-fa-relation}, we get
$$(f_1-f_2)\cdot s'=x\cdot s\cdot s'=(f'_1-f'_2)\cdot s,$$
or equivalently,
$$sf'_1-s'f_1=sf'_2-s'f_2.$$
Hence, we can rewrite the above formula as
$$\lan [s,s'],v\ran=2 v\cdot(sf'_1-s'f_1).$$

It remains to compare this with $p (AA'-A'A)i(1)$. We can compute this over $U_1$:
$$p(AA'-A'A)i(1)=pA[i(f'_1)+\si_1(s')]-pA'[i(f_1)+\si_1(s)]=2(f'_1s-f_1s').$$
Comparing this with the formula above, we get
$$\lan p(AA'-A'A)i(1),v\ran=\lan [s,s'],v\ran$$
which proves our assertion.
\end{proof}

\begin{rmk}
Proposition \ref{conormal-Lie-prop} is a part of a broader picture, involving the notion of a symplectic groupoid (see \cite{Wein}). Given
a symplectic groupoid $M$ with the space of objects $X$ (in the category of smooth schemes), then $X$ has a natural Poisson structure. In this situation 
one can also consider the algebraic stack $\XX$ associated with the groupoid $M$. Then the Lie algebra of automorphisms in $M$ of an object $x\in X$ 
(or equivalently, of the corresponding point of $\XX$) is naturally identified with the coisotropic Lie algebra at $x\in X$ of the Poisson structure on $X$ (see \cite[ch.\ 2, Lem.\ 1.2]{KarMas}).

In the case of the Feigin--Odesskii Poisson structure on $X=\P\Ext^1(L,\OO)$, the symplectic groupoid $M$ can be constructed as follows. We have a 
natural map $f:X\to \Bun_L$, where $\Bun_L$ is the stack of rank $2$ bundles on $E$ with
the determinant $L$, and $f$ associates the bundle $V$ with an extension
\eqref{Ext-VL-eq}. Then we define $M$ to be the $2$-fibered product 
$$M:=X\times_{\Bun_L} X.$$ 
The first part of the proof of Prop.\ \ref{conormal-Lie-prop} shows that the map $f$
induces a surjection on tangent spaces, hence, $M$ is a smooth algebraic space. From the fibered product structure we get a groupoid structure on $M$.
The corresponding stack $\XX$ is equipped with a fully faithful map
$\XX\to \Bun_L$ (see \cite[Lem.\ 93.16.1]{stacks}, so automorphisms of points are the same.
One can show that in fact $M$ has a structure of a symplectic groupoid that induces the Feigin--Odesskii Poisson structure on $X$ (the details will appear elsewhere). 
This gives a more conceptual proof of the identification of Lie algebras in Prop. \ref{conormal-Lie-prop}.
\label{rmk1}
\end{rmk}

It is well known that the FO-bracket $\Pi_E$ vanishes at any $p\in E$. Hence, in this case we get a Lie algebra structure on $T_p^*\P^n$.

\begin{cor}\label{conormal-Lie-cor}
For $p\in E\sub \P^n$, the conormal Lie algebra $\fg=T_p^*\P^n$ admits a basis $x_1,\ldots,x_{n-1},y$, such that 
$\lan x_1,\ldots,x_{n-1}\ran=(T_xE)^\perp$ is an abelian subalgebra, and $[y,x_i]=x_i$.
\end{cor}

\begin{proof}
In this case $V=\OO(p)\oplus L(-p)$ and the algebra $\End(V)/\lan\id\ran$ has an abelian subalgebra $A$ corresponding to endomorphisms of
the form $V\to \OO(p)\to L(-p)\to V$. The $1$-dimensional complement to $A$ is spanned by the idempotent $e$ corresponding to the summand $L(-p)$, and for any $a\in A$
we have $ea=a$, $ae=0$, so $[e,a]=a$.
\end{proof}

Using this description of the conormal Lie algebras, we get the following constraints on Poisson brackets compatible with $\Pi_E$.

\begin{prop}\label{bracket-E-prop} 
Let $E\sub \P^n$ be a normal elliptic curve, and let $\Pi$ be a Poisson bracket on $\P^n$. Assume that  $[\Pi_{E},\Pi]=0$.
Then for every $p\in E$, one has $\Pi|_p=v_1\we v_2$, where $v_1\in T_pE$.
\end{prop}

\begin{proof} 
We will calculate $[\Pi_E,\Pi]|_p$ for $p\in E$, and then equate it to zero. We have $\Pi_E|_p=0$, and 
the linear part of $\Pi_{E}$ is 
$$\Pi_{E}^{lin}:=\sum_i x_i\partial_{x_i}\we\partial_y,$$ 
where $x_1,\ldots,x_{n-1},y$ is a basis of $T_p\P^n$ chosen as in Corollary \ref{conormal-Lie-cor}.
The bracket $[\Pi_E,\Pi]|_p=[\Pi_E^{lin},\Pi]|_p$ depends only on $\Pi|_p$. We have 
$$[\Pi_{E}^{lin},\partial_{x_i}\we\partial_y]=0,$$
$$[\Pi_{E}^{lin},\partial_{x_i}\we\partial_{x_j}]=\pm 2\partial_{x_i}\we\partial_{x_j}\we\partial_y.$$
Hence, $[\Pi_{E},\Pi]|_p=0$ if and only if $\Pi|_p=v\we \partial_y$, for some $v\in T_p\P^n$. 
Since $\lan\partial_y\ran=T_pE$, our assertion follows.
\end{proof}

\subsection{Secant variety of $E$ and its fibration in scrolls}\label{scrolls-sec}

Assume $n\ge 4$, and let $E\sub \P^n$ be a normal elliptic curve, $L=\OO(1)|_E$. Consider the $\P^1$-bundle $p:X\to E^{[2]}$ over the symmetric square of $E$, 
with the fiber 
$\P H^0(L|_D)^*\sub \P H^0(E,L)^*=\P^n$ over a point $D\in E^{[2]}$. We have $X=\P(\VV)$, for the following vector bundle $\VV$ on $E^{[2]}$.
Consider the natural map $\pi:E\times E\to E^{[2]}$, then $\VV=(\pi_*p_1^*L)^\vee$, where $p_1:E^2\to E$ is the projection to the first factor.

We have a well defined morphism $\phi:X\to \P^n$, embedding linearly each line $p^{-1}(D)\sub X$ into
$\P^n$, so that its image is precisely the corresponding secant line corresponding to $D$. More precisely, the natural morphism $H^0(L)\ot \OO_{E^2}\to p_1^*L$
induces a surjection $H^0(L)\to \pi_*p_1^*L=\VV^{\vee}$. Hence, we get a closed embedding $X=\P(\VV)\hra E^{[2]}\times \P H^0(L)^*$, and $\phi$ is obtained
by the second projection.

\begin{lem} Assume that $n\ge 4$.
The map $\phi:X\to \P^n$ induces a surjective birational morphism $X\to \Sec^2(E)$ which is an isomorphism over $\Sec^2(E)\setminus E$.
\end{lem} 

\begin{proof} This is a well known consequence of the fact that $L$ ``separates $4$ points on $E$" since $\deg(L)\ge 5$ (see Terracini's Lemma as stated in \cite[Lem.\ 1.2, Lem.\ 1.4]{Bertram}).
\end{proof}

Next, we consider the Abel-Jacobi map $a:E^{[2]}\to \Pic_2(E): D\mapsto \OO(D)$. For each $M\in \Pic_2(E)$, we set $S_M:=\pi^{-1}(a^{-1}(M))\sub X$.
Note that $\phi(S_M)\sub\P^n$ is the union of secant lines $\ov{pp'}$ over all pairs $p,p'\in E$ such that $\OO_E(p+p')\simeq M$.

We refer to \cite{EH} or \cite{Reid} for background on rational normal scrolls.

\begin{lem} Assume that $n\ge 4$. Then the restriction $\phi|_{S_M}:S_M\to \P^n$ is an embedding. If $n=2r$, then $\phi(S_M)$ is the scroll of type $S(r-1,r)$.
If $n=2r+1$, then $\phi(S_M)$ is the scroll of type $S(r,r)$ for $M$ such that $M^{r+1}\not\simeq L$, and it is of type $S(r-1,r+1)$ for $M$ such that $M^{r+1}\simeq L$.
\end{lem}

\begin{proof} Let us compute the restriction of the bundle $\VV$ to the projective line $a^{-1}(M)\sub E^{[2]}$. By the base change formula,
$$\VV|_{a^{-1}(M)}\simeq \pi_{M*}(p_1^*L|_{E_M}),$$
where $E_M\sub E^2$ is the preimage of $a^{-1}(M)$ under $\pi:E^2\to E^{[2]}$, $\pi_M:E_M\to a^{-1}(M)$ is the projection. 
Thus, $E_M=\{(x,y)\in E^2 \ | \ \OO(x+y)\simeq M\}$.
Note that the projection $p_1:E_M\to E$ is an isomorphism, and the composition $\pi_M(p_1|_{E_M})^{-1}:E\to a^{-1}(M)$ is the double cover that can be identified with
the map $f:E\to \P^1$ given by the linear system $|M|$. We want to compute the splitting type of $f_*L$.

Assume first that $n=2r$, so $\deg(L)=2r+1$. We claim that in this case $f_*L\simeq \OO(r-1)\oplus \OO(r)$. Indeed, this follows from the fact that
$\deg(f_*L)=\deg(L)-2=2r-1$ and from the vanishing $H^0(f_*L(-r-1))=H^0(L\ot M^{-r-1})=0$ (since $\deg(L\ot M^{-r-1})=-1$).

Now assume that $n=2r+1$, so $\deg(L)=2r+2$. 
If $L\not\simeq M^{r+1}$ then $H^0(f_*L(-r-1))=H^0(L\ot M^{-r-1})=0$, so $f_*L\simeq \OO(r)\oplus \OO(r)$. On the other hand, if $L\simeq M^{r+1}$,
then $H^0(f_*(-r-1))$ is $1$-dimensional, so $f_*L\simeq \OO(r-1)\oplus \OO(r+1)$.

This proves that $S_M$ is isomorphic to the projectivization of the rank 2 vector bundle over $a^{-1}(M)\simeq \P^1$ of one of the types $\OO(n_1)\oplus \OO(n_2)$ described above.
The fact that the corresponding morphism $S_M\to \P^n$ is an embedding is the standard fact (using that $n_1>0$ and $n_2>0$, which is the case since $n\ge 4$).
\end{proof}

\begin{cor} Assume that $n\ge 4$. There is a well defined morphism $\Sec^2(E)\setminus E\to \Pic_2(E)\simeq E$, whose fibers are of the form $\phi(S_M)\setminus E$, where $\phi(S_M)$
is the scroll of type $S(r-1,r)$ for $n=2r$, or of type $S(r,r)$ or $S(r-1,r+1)$ for $n=2r+1$.
\end{cor}

\subsection{Rank $2$ locus}

Fix a normal elliptic curve $E\sub\P^n$, where $n\ge 4$, and let $\Pi_E$ be the corresponding FO-bracket on $\P^n$.
We want to describe the locus of points in $\P^n$ where the rank of $\Pi_E$ is $\le 2$.

Let $L:=\OO(1)|_E$, so we can identify the embedding with $E\to \P H^0(L)^*\simeq \P\Ext^1(L,\OO)$.
We will use the following well known fact (see \cite[Prop.\ 2.3]{HP-bih}): let $x\in \Ext^1(L,\OO)\setminus \{0\}$ be the class of an extension \eqref{Ext-VL-eq}.
Then the rank of $(\Pi_E)_x$ is equal to $\deg(L)-\dim\End(V)$.

In the case $n=5$, i.e., $L:=\OO(1)|_E$ has degree $6$, for each $M\in \Pic_3(E)$ such that $M^2\simeq L$ (there are four such $M$), 
we define a Veronese surface $P_M\sub \P^n$ as follows. We observe that the natural map
$$S^2H^0(E,M)\to H^0(E,M^2)\simeq H^0(E,L)$$
is an isomorphism, and define $P_M$ as the image of the Veronese embedding
$$\P H^0(M)^*\to \P H^0(S^2 H^0(M)^*)\simeq \P H^0(L)^*=\P^n.$$

\begin{prop}\label{rk2-prop} (i) Assume $n=4$ or $n\ge 6$. Then for $x\in \P^n$, one has $\rk (\Pi_E)_x\le 2$ if and only if $x\in \Sec^2(E)$.

\noindent
(ii) Assume $n=5$. Then for for $x\in \P^n$, one has $\rk (\Pi_E)_x\le 2$ if and only if $x\in \Sec^2(E)$ or $x\in P_M\sub \P^5$,
for some $M\in \Pic_3(E)$ such that $M^2\simeq L$.
\end{prop}

\begin{proof} (i) Let $\Pi=\Pi_E$. By \cite[Prop.\ 2.3]{HP-bih}, we have $\rk \Pi_x\le 2$ if and only if $\dim\End(V)\ge \deg(L)-2$, where $\deg(L)=n+1$.
Assume first that $n\ge 6$. Then we get $\dim \End(V)\ge 5$, so $V$ cannot be semistable. Hence, $V$ is isomorphic to a direct sum of line bundles:
$V=L_1\oplus L_2$, where $\deg(L_1)<\deg(L_2)$. Furthermore, we have 
$$\dim \End(V)=2+\deg(L_2)-\deg(L_1)\ge \deg(L)-2=\deg(L_1)+\deg(L_2)-2,$$
so $\deg(L_1)\le 2$. The composed map $\OO\to V\to L_1$ cannot be zero, so it vanishes on an effective divisor $D$ of degree $\le 2$. Since $V$ is a nontrivial extension,
we have $\deg(D)\ge 1$. If $\deg(D)=1$ then our extension splits under $\OO\to \OO(p)$ for some point, so $x$ lies on the elliptic curve $E\sub\P^n$. If $\deg(D)=2$ then our
extension splits under $\OO\to \OO(D)$, so it lies on the chord of $E$ in $\P^n$ corresponding to $D$.

Conversely, if $x\in \Sec^2(E)$, then there exists an effective degree $2$ divisor $D$ on $E$ such that the extension associated with $x$ 
splits under $\OO\to \OO(D)$. If $x\in E$ then the extension splits already under $\OO\to \OO(q)$ for some $q\in E$. Then $V\simeq \OO(q)\oplus L(-q)$,
so $\dim \End(V)=\deg(L)$, which means that $\Pi_x=0$. Otherwise, if $x\in \Sec^2(E)\setminus E$, then we get an exact sequence 
$$0\to L(-D)\to V\to \OO(D)\to 0.$$
Since $\deg(L)\ge 5$, we have $\deg(L(-D))=3>2=\deg(\OO(D))$, so $V\simeq \OO(D)\oplus L(D)$. Hence, $\deg \End(V)=\deg(L)-2$ and $\rk \Pi_x=2$.

Assume now that $n=4$. Then the extensions $V$ corresponding points where $\rk \Pi_x\le 2$ satisfy $\dim \End(V)\ge 3$. Since $\deg(V)=5$, this means that $V$ is unstable,
so $V=L_1\oplus L_2$ where $(\deg(L_1),\deg(L_2))$ is either $(1,4)$ or $(2,3)$. We finish the proof in the same way as before.

\noindent
(ii) Now the condition on $V$ (which has degree $6$) is that $\dim \End(V)\ge 4$. This means that either $V$ is unstable, or $V=M\oplus M$, for some line bundle $M$ of degree $3$,
such that $M^2\simeq L$. In the former case, we use the same argument as in part (i) to show that the correspond point lies on $\Sec^2(C)$. We claim that extensions
with $V\simeq M\oplus M$ correspond to points of $P_M\setminus E$, for the Veronese surface $P_M\sub \P H^0(L)^*$. Indeed, an extension of the form
$$0\to \OO\to M\oplus M\rTo{p} L\to 0$$
corresponds to a base point free pencil $\lan s_1,s_2\ran \sub H^0(M)$, which corresponds to a point $\ell:=\lan s_1,s_2\ran^\perp\sub H^0(M)^*$ of the projective plane $\P H^0(M)^*$.
The line in $\P H^0(L)^*$ corresponding to our extension is given by the hyperplane $H\sub H^0(L)$, obtained as the image of the map $H^0(M)^2\to H^0(L)$ induced by $p:M\oplus M\to L$. In other words, 
$$H=s_1\cdot H^0(M)+s_2\cdot H^0(M)\sub H^0(L).$$
It remains to observe that $H$ is exactly the space of quadrics on $\P H^0(M)^*$ vanishing at the point $\ell$. Hence, it is exactly the image of $\ell$ under the Veronese embedding
$\P H^0(M)^*\to \P S^2H^0(M)^*\simeq \P H^0(L)^*$.
\end{proof}

\begin{lem}\label{tangent-map-lem}
Assume $n\ge 4$, and let $x\in \Sec^2(E)\setminus E$. Let $D\sub E$ be the corresponding effective divisor of degree $2$. Then
we can view $x$ as an element of the projective line $\P H^0(L^{-1}(D)|_D)\sub \P H^1(L^{-1})$, and its lifting $\wt{x}\in H^0(L^{-1}(D)|_D)$
gives a trivialization of $L^{-1}(D)|_D$. The tangent space to $\Sec^2(E)$ at $x$ can be identified with $H^0(L^{-1}(2D)|_{2D})/\lan \wt{x}\ran$, so that the tangent map 
to the morphism $\Sec^2(E)\setminus E\to \Pic_2(E)$ is given by the composition
$$H^0(L^{-1}(2D)|_{2D})/\lan \wt{x}\ran\to H^0(L^{-1}(2D)|_D)\rTo{\wt{x}^{-1}} H^0(\OO(D)|_D)\to H^1(\OO).$$
\end{lem}

\begin{proof}
This is a standard calculation related to Terracini Lemma (see \cite[Sec.\ 1]{Bertram}).
\end{proof}

Recall that for $n\ge 4$, the secant variety $\Sec^2(E)$ has a fibration in scrolls $S_M$ over $E$, described in Sec.\ \ref{scrolls-sec}.

\begin{lem}\label{image-Pi-Sec-lem} (i) Assume $n\ge 4$. For $x\in \Sec^2(E)\setminus E$, one has 
$$\im((\Pi_E)_x:T^*_x\P^n\to T_x\P^n)=T_xS_M,$$
where $S_M$ is the scroll associated with $M\in\Pic_2(E)$, passing through $x$.

\noindent
(ii) Assume $n=3$. Let $\Si_E\sub \P^3$ denote the set of four vertices of singular quadrics passing through $E$.
For $x\in \P^3\setminus (E\cup \Si_E)$, one has 
$$\im((\Pi_E)_x:T^*_x\P^n\to T_x\P^n)=T_xQ,$$ where $Q$ is the unique smooth quadric
passing through $E$ and $x$.
\end{lem}

\begin{proof}
(i) Set $\Pi=\Pi_E$. We have $\im(\Pi_x)=\ker(\Pi_x)^\perp$, and as we have seen in the proof of Proposition \ref{conormal-Lie-prop}, $\ker(\Pi_x)$ is precisely the image of the map
\begin{equation}\label{EndV-pAi-emb}
\End(V)/\lan\id\ran\to \lan x\ran^\perp\sub H^0(L): A\mapsto p\circ A\circ i,
\end{equation}
where $x$ corresponds to an extension \eqref{Ext-VL-eq}.  

As we have seen in the proof of Proposition \ref{rk2-prop}, $x\in \Sec^2(E)\setminus E$ if and only if $V\simeq \OO(D)\oplus L(-D)$, where $D$ is an effective divisor of degree $2$ and the
embedding $i:\OO\to \OO(D)\oplus L(-D)$ is given by $(1,s)$, for some section $s\in H^0(L(-D))$ not vanishing on $D$. The projection
$p:\OO(D)\oplus L(-D)\to L$ is given by $(s,-1)$. 
Thus, $\End(V)/\lan\id\ran$ is generated by the idempotent $e$ corresponding to the factor $\OO(D)$, and by nilpotent endomorphisms induced by maps $t:\OO(D)\to L(-D)$.
The image of $e$ under the embedding \eqref{EndV-pAi-emb} is $s\in H^0(L(-D))\sub H^0(L)$, while the image of $t\in H^0(L(-2D))$ is $-t\in H^0(L(-2D))\sub H^0(L)$.
In other words, $\ker(\Pi_x)$ is identified with the preimage of the line $L$ spanned by $s|_{D}\in H^0(L(-D)|_D)\sub H^0(L|_{2D})$ under the projection $H^0(L)\to H^0(L|_{2D})$.
It follows that $\im(\Pi_x)=\ker(\Pi_x)^\perp$ is the image of $L^\perp\in H^0(L^{-1}(2D)|_{2D})$ under the natural homomorphism $H^0(L^{-1}(2D)|_{2D})\to H^1(L^{-1})/\lan x\ran$.

Note that the restriction $s|_D$ gives a trivialization of $L(-D)|_D$. We claim that its inverse $(s|_D)^{-1}\in H^0(L^{-1}(D)|_D)$ induces the class of our extension under the connecting homomorphism
$H^0(L^{-1}(D)|_D)\to H^1(L^{-1})$. 
Indeed, this is equivalent to checking that our extension \eqref{Ext-VL-eq} is obtained as the pullback of the standard extension
$$0\to \OO\to \OO(D)\to \OO_D(D)\to 0$$
under the map $(s|_D)^{-1}:L\to \OO_D(D)$.
But this immediately follows from the existence of the following morphism of exact sequences
\begin{diagram}
0&\rTo{}& \OO &\rTo{}& V &\rTo{}& L&\rTo{}& 0\\
&&\dTo{\id}&&\dTo{p_{\OO(D)}}&&\dTo{(s|_D)^{-1}}\\
0&\rTo{}& \OO &\rTo{}& \OO(D) &\rTo{}& \OO_D(D) &\rTo{}& 0
\end{diagram}
where the middle arrow is the projection onto the summand $\OO(D)\sub V$.

Hence, by Lemma \ref{tangent-map-lem}, the tangent map to $\Sec^2(E)\setminus E\to \Pic_2(E)$ at $x$ can be identified with the composition
$$H^0(L^{-1}(2D)|_{2D})/\lan \wt{x}\ran\to H^0(L^{-1}(2D)|_D)\rTo{s|_D} H^0(\OO(D)|_D)\to H^1(\OO).$$
But this composition is precisely the pairing with $s|_D\in H^0(L|_{2D})$, so $L^\perp$ goes to zero. 

\noindent
(ii) This immediately follows from the well known fact that the quadrics passing through $E$ are Poisson subvarieties of $\P^3$.
\end{proof}

\begin{rmk}
Lemma \ref{image-Pi-Sec-lem} also follows easily from the fact that the symplectic leaves of the FO brackets correspond to fixing the isomorphism class of an extension, as mentioned
in Remark \ref{rmk1}.
\end{rmk}





\section{Proofs of the main results}

\subsection{Proof of Theorem A}

Consider first the case of two brackets. Assume $[\Pi_{E_1},\Pi_{E_2}]=0$.  By Proposition \ref{bracket-E-prop}, at every point
$p\in E_1\setminus (E_2\cup\Si_{E_2})$ (where $\Si_{E_2}$ is the set of four vertices of singular quadrics through $E_2$), one has $(\Pi_{E_2})_p=v_1\we v_2$, where $v_1\in T_pE_1$. 
Note also that the subspace $\lan v_1,v_2\ran\sub T_p\P^3$ is exactly the image of $(\Pi_{E_2})_p:T_p^*\P^3\to T_p\P^3$, which is contained
in $T_pQ$ where $Q$ is the unique quadric passing through $p$ and $E_2$ (see Lemma \ref{image-Pi-Sec-lem}(ii)). 
Hence, we get an inclusion $T_pE_1\subset T_pQ$.
It follows that the restriction to $E_1\setminus E_2$ of the map
$\P^3\setminus E_2\to \P^1$ given by quadrics through $E_2$, has zero tangent map at every point. Hence, this map is constant,
 which means that $E_1$ is contained in a single quadric through $E_2$.


\medskip


For the case of a general collection $(E_i)$, let $(L_i=L(E_i))$ denote the corresponding collection of projective lines in the projective space of all quadrics (where
$L(E_i)$ is the pencil of quadrics passing through $E_i$). Then $(\Pi_{E_i})$ are compatible if and only if every pair of lines $L_i$ and $L_j$ has nontrivial intersection.
It is well known that this happens if and only if either all lines pass through one point, or all are contained in a projective plane.

Conversely, the fact that the FO brackets associated with anticanonical divisors on a quadric in $\P^3$ are compatible follows from the results of \cite{HP-bih} 
(see \cite[Ex.\ 4.8]{HP-bih}).
\qed

\subsection{Proof of Theorem B}

Consider first the case when we have two normal elliptic curves $E_1,E_2\sub \P^n$ such that $[\Pi_{E_1},\Pi_{E_2}]=0$.  By Proposition \ref{bracket-E-prop}, at every point
$p\in E_1$, one has $(\Pi_{E_2})_p=v_1\we v_2$, where $v_1\in T_pE_1$. In particular, $\rk (\Pi_{E_2})_x\le 2$, so by Proposition \ref{rk2-prop},
we have an embedding $E_1\sub \Sec^2(E_2)$, or (in the case $n=5$), $E_1\sub P_M$, where $P_M\sub \P^5$ is the Veronese surface associated with some square
root $M$ of $L$. In the latter case we are done, so we can assume that $E_1\sub \Sec^2(E_2)$.
Using Lemma \ref{image-Pi-Sec-lem}(i), for $p\in E_1\setminus E_2$, we obtain the inclusion $T_pE_1\sub T_pS_M$, where $S_M\sub \Sec^2(E_2)$ is some scroll, which is the
(closure of the) fiber of the map $\Sec^2(E_2)\setminus E_2\to E_2$. Thus, the restriction of the latter map to $E_1$ is constant, so $E_1$ is contained in some $S_M$,
as claimed.

Next, consider any family $(E_i)$ such that $(\Pi_{E_i})$ are compatible.
Consider a pair $E_1\neq E_2$ in this family, and let $E_i$ be any other elliptic curve in the family. 
By the first part of the proof, there exists a rational surface $S\sub \P^n$ (either a scroll or a Veronese surface), such that $E_1$ and $E_2$ are both anticanonical divisors on $S$.
Note that the linear combinations of $\Pi_{E_1}+\la \Pi_{E_2}$ are the FO brackets associated with elliptic curves $E_\la$ in the pencil of anticanonical divisors on $S$,
generated by $E_1$ and $E_2$. Hence, for each $E_\la$ the brackets $\Pi_{E_\la}$ and $\Pi_{E_i}$ are compatible. By the first part of the proof, this implies that
each $E_\la$ is contained in $\Sec^2(E_i)$. But the surface $S$ is the closure of the union $\cup_\la E_\la$, hence we get the inclusion $S\sub \Sec^2(E_i)$.
Since $S$ is rational, the composed map $S\setminus E_i\to \Sec^2(E_i)\to E_i$ is constant. Therefore, $S$ is contained in the closure $S'$ of the fiber of the map 
$\Sec^2(E_i)-E_i\to E_i$, which is one of the scrolls in $\SS_{E_i}$. Since $S'$ is an irreducible surface, we get $S=S'$, as claimed.

Conversely, the fact that the FO brackets associated with anticanonical divisors on a scroll or on a Veronese surface in $\P^5$ are compatible follows from the results of \cite{HP-bih} 
(see \cite[Ex.\ 4.6, 4.8]{HP-bih}).
\qed

\medskip

\begin{proof}[Proof of Corollary C] First, we recall that by \cite[Cor.\ 1.2]{PS}, if $n$ is even then for any FO-bracket $\Pi_E$ on $\P^n$, 
there exists a Zariski open neighborhood of $\Pi_E$ in the variety of Poisson brackets consisting of FO-brackets. Hence, for even $n$, any linear subspace of Poisson brackets containing
$\Pi_E$ is an FO-subspace. Thus, we only need to prove our assertions for FO-subspaces. 
Now part (b) is immediate from Theorem B. Part (a) also follows Theorem B, since if all elliptic curves lying on $S$, which is either or scroll or a Veronese surface,
are contained in another irreducible surface $S'$, then $S=S'$.
\end{proof}

\subsection{Proof of Theorem D}

Let $\P^3=\P V$, where $V$ is a $4$-dimensional vector space. Since normal elliptic curves in $\P^3$ are intersections of pairs of quadrics,
we can view the construction of the FO Poisson bracket as associating with a pair of generic quadratic forms $Q_1,Q_2\in S^2V^*$, a Poisson bivector 
$\Pi_{Q_1=Q_2=0}\in H^0(\P V,{\bigwedge}^2 T)$ defined up to rescaling. 
More precisely, the choice of quadrics $Q_1,Q_2$ and of a volume form on $V$ gives a nonvanishing differential on the elliptic curve $E=(Q_1=0)\cap (Q_2=0)$,
which is used in the construction of $\Pi(Q_1,Q_2)$. We claim that in fact there is a linear $\GL(V)$-equivariant map
$$\Psi:{\bigwedge}^2(S^2V^*)\to \det(V^*)\ot H^0(\P V,{\bigwedge}^2 T)$$
such that for $Q_1$ and $Q_2$ defining an elliptic curve, $\Pi(Q_1,Q_2)$ coincides with $\Psi(Q_1\we Q_2)$ up to a factor.
Indeed, it is well known that the FO brackets associated with the elliptic curve $Q_1=Q_2=0$ are induced by the quadratic Poisson bracket on $S(V^*)$ given by
$$\{\ell_1,\ell_2\}=d\ell_1\we d\ell_2\we dQ_1\we dQ_2/\vol_V^*,$$
where $\vol_V^*$ is a generator of $\det(V^*)$ (see e.g., \cite[Sec.\ 2.1]{Rubtsov}). This gives the required linear map $\Psi$ which should be viewed as a normalized version of the bivectors $\Pi_{Q_1=Q_2=0}$.

Hence, taking the Schouten bracket of two FO brackets corresponds to a linear $\GL(V)$-invariant map
$$\Psi^{(2)}:({\bigwedge}^2(S^2V^*))^{\ot 2}\to \det(V^*)^2\ot H^0(\P V,{\bigwedge}^3 T)\simeq \det(V^*)\ot S^4(V^*).$$
We claim that this map factors through the multiplication map in the exterior algebra of $S^2V^*$,
$$\mu:({\bigwedge}^2(S^2V^*))^{\ot 2} \to {\bigwedge}^4(S^2V^*).$$

The proof will be based on the fact that by Theorem A, $\Psi^{(2)}((Q_1\we Q_2)\ot (Q_3\we Q_4))=0$ whenever the planes $\lan Q_1,Q_2\ran$ and $\lan Q_3,Q_4\ran$ in $S^2(V^*)$ have
nonzero intersection (since the corresponding FO brackets are compatible). Set $W:=S^2(V^*)$.
Namely, we can interpret the dual map $(\Psi^{(2)})^\vee$ as a section of the vector bundle $\det(V^*)\ot S^4(V^*)\ot \OO(1,1)$ on the square of the Grassmannian variety 
$G(2,W)^2$, which vanishes on the subvariety $Z\sub G(2,W)^2$ consisting of pairs of planes $P_1,P_2\sub W$ with $P_1\cap P_2\neq 0$.
It is well known that the subspace 
$$H^0(G(2,W)^2,\II_Z(1,1))\sub H^0(G(2,W)^2,\OO(1,1))$$
is precisely the image of the map 
$$\mu^\vee:{\bigwedge}^4(W)^*\to ({\bigwedge}^2(W)^*)^{\ot 2}$$
dual to the multiplication. Hence, $(\Psi^{(2)})^\vee$ belongs to the subspace $\det(V^*)\ot S^4(V^*)\ot \im(\mu^\vee)$, or equivalently, $\Psi^{(2)}$ factors through $\mu$.

Thus, we get $\Psi^{(2)}=\Phi'\circ \mu$ for some nonzero $\GL(V)$-invariant map
$$\Phi':{\bigwedge}^4(S^2V^*)\to \det(V^*)\ot S^4(V^*).$$
But $\det(V^*)\ot S^4(V^*)$ is an irreducible $\GL(V)$-representation occurring with multiplicity one in ${\bigwedge}^4(S^2V^*)$, Hence, $\Phi'=c\Phi$ for some constant $c$.

Finally, to find the constant $c$, we calculate both sides for 
$$Q_1=x_1^2, \ Q_2=x_1x_2, \ Q_3=x_1x_3, \ Q_4=x_1x_4.$$
We use coordinates $y_i=x_i/x_1$, $i=2,3,4$ on the affine part of $\P^3$.
The Poisson structure associated with $(Q_1,Q_2)$ has the only nonzero bracket between the coordinates
$$\{y_3,y_4\}_1=2$$
(where we trivialize $\det(V^*)$ using $\vol=x_1\we x_2\we x_3\we x_4$), so 
$$\Pi_1=2\partial_{y_3}\we \partial_{y_4}.$$
The Poisson structure associated with $(Q_3,Q_4)$ has 
$$\{y_2,y_3\}=2y_3,\ \{y_2,y_4\}=2y_4, \ \{y_3,y_4\}=0,$$
so
$$\Pi_2=2y_3\cdot \partial_{y_2}\we\partial_{y_3}+2y_4\cdot\partial_{y_2}\we\partial_{y_4}.$$
Therefore, 
$$[\Pi_1,\Pi_2]=8\partial_{y_2}\we\partial_{y_3}\we\partial_{y_4},$$
which corresponds to the quartic $8x_1^4$.
On the other hand the quartic polynomial associated with $Q_1\we Q_2\we Q_3\we Q_4$ sends $e_1$ to 
$2\cdot\vol$. Hence, $c=4$. 
\qed

\end{document}